\documentclass[12pt]{article}
\usepackage{amscd,amsmath,amssymb,amsfonts,braket,mathrsfs,latexsym,enumerate,amsthm}
\usepackage[all]{xy}
\makeatletter
\@namedef{subjclassname@2010}{%
  \textup{2010} Mathematics Subject Classification}
\makeatother

\newtheorem{thm}{Theorem}
\newtheorem{lem}[thm]{Lemma}

\newtheorem{prop}[thm]{Proposition}

\theoremstyle{definition}

\numberwithin{thm}{section}
\numberwithin{equation}{section}

\newcommand{\C}{\mathbb{C}}
\newcommand{\F}{\mathbb{F}}

\newcommand{\Q}{\mathbb{Q}}

\newcommand{\Kb}{\overline{K}}
\newcommand{\kb}{\overline{k}}

\newcommand{\Kab}{K^{\text{ab}}}

\newcommand{\Gab}{\mathrm{G}^{\text{ab}}}

\newcommand{\varrhoab}{\varrho^{\text{ab}}}
\newcommand{\varphiab}{\varphi^{\text{ab}}}
\newcommand{\chiab}{\chi^{\text{ab}}}

\newcommand{\A}{\mathbb{A}}

\newcommand{\R}{\mathbb{R}}
\newcommand{\Z}{\mathbb{Z}}

\newcommand{\betab}{\overline{\beta}}
\newcommand{\pib}{\overline{\pi}}
\newcommand{\gammab}{\overline{\gamma}}

\newcommand{\mfa}{\mathfrak{a}}

\newcommand{\mfp}{\mathfrak{p}}
\newcommand{\mfq}{\mathfrak{q}}

\newcommand{\mfP}{\mathfrak{P}}
\newcommand{\mfQ}{\mathfrak{Q}}

\newcommand{\Dtil}{\widetilde{D}}

\newcommand{\util}{\widetilde{u}}
\newcommand{\wtil}{\widetilde{w}}

\newcommand{\betatil}{\widetilde{\beta}}

\newcommand{\cA}{\mathcal{A}}

\newcommand{\cC}{\mathcal{C}}

\newcommand{\cP}{\mathcal{P}}

\newcommand{\cO}{\mathcal{O}}

\newcommand{\cN}{\mathcal{N}}

\newcommand{\cS}{\mathcal{S}}
\newcommand{\cT}{\mathcal{T}}

\newcommand{\Ram}{\mathbf{Ram}}

\newcommand{\Gal}{\mathrm{Gal}}
\newcommand{\G}{\mathrm{G}}
\newcommand{\M}{\mathrm{M}}
\newcommand{\N}{\mathrm{N}}

\newcommand{\End}{\mathrm{End}}
\newcommand{\Aut}{\mathrm{Aut}}

\newcommand{\Norm}{\mathrm{Norm}}

\newcommand{\Nrd}{\mathrm{Nrd}}

\newcommand{\tr}{\mathrm{tr}}

\newcommand{\Tr}{\mathrm{Tr}}

\newcommand{\cf}{cf.\ }

\newcommand{\resp}{resp.\ }

\newcommand{\bq}{\beta_{\mfq}}

\newcommand{\one}{\pmb{1}}

\begin{document}

\title{Points on Shimura curves rational over imaginary quadratic fields
in the non-split case}
\author{Keisuke Arai\\
Department of Mathematics,
School of Science and Technology for Future Life,\\
Tokyo Denki University\\
5 Senju Asahi-cho, Adachi-ku, Tokyo 120-8551 Japan\\
E-mail: araik@mail.dendai.ac.jp}


\date{}


\maketitle



\begin{abstract}
For an imaginary quadratic field $k$ of class number $>1$,
we prove that there are only finitely many isomorphism classes of
rational indefinite quaternion division algebras $B$ such that
the associated Shimura curve $M^B$ has $k$-rational points. 
In other words, the main result asserts that
there is a finite set $P(k)$ of prime numbers depending on $k$ such that:
if there is a prime divisor of the discriminant of $B$ which is not in $P(k)$,
then $M^B$ has no $k$-rational points.
Moreover, we can take $P(k)$ to satisfy the following:
There is an effectively computable constant $C(k)$ depending on $k$ such that
$p\in P(k)$ implies $p<C(k)$
with at most one possible exception.

The case where $k$ splits $B$ was done by Jordan.
In the non-split case, the proof is done by studying a canonical isogeny
character and its composition with the transfer map.
\end{abstract}

\noindent
{\bf keywords:}
rational points, Shimura curves, QM-abelian surfaces



\section{Introduction}
\label{intro}

For an integer $N\geq 1$,
let $X_0(N)$ be the smooth compactification of the coarse moduli scheme over $\Q$
parameterizing the isomorphism classes of $(E,C)$, where 
\begin{itemize}
\item
$E$ is an elliptic curve, and 
\item
$C$ is a cyclic subgroup of $E$ of order $N$.
\end{itemize}
Then $X_0(N)$ is a proper smooth curve over $\Q$, which is called a 
\textit{modular curve}.
If $p$ is a prime number $>163$, then the set $X_0(p)(\Q)$ of $\Q$-rational points 
consists of only two cusps
(see \cite[Theorem 7.1]{Ma}).
This theorem was extended to a result on rational points over a quadratic field
which is not an imaginary quadratic field of class number one
(see \cite[Theorem B]{Mo}).
These results can be interpreted as follows.

\begin{description}
\item
\textit{
If the level of the modular curve grows, then the set of rational points on it 
becomes small.}
\end{description}
A similar result is found concerning the quotient 
$X_0^+(p^r):=X_0(p^r)/w_{p^r}$ ($p\geq 11$, $p\ne 13$ and $r\geq 2$; or $p=13$ and $r=2$)
by the Atkin-Lehner involution,
in which case a $\Q$-rational point is either a CM point or a cusp
(see \cite[Theorem 1.1]{BDMTV}, \cite[Theorem 1.2]{BP}, \cite[Theorem 1.1]{BPR}).

Let $B$ be an indefinite quaternion division algebra over $\Q$, and let $d(B)$ be its
discriminant.
Then $d(B)>1$, and it is the product of an even number of distinct prime numbers.
Fix a maximal order $\cO$ of $B$.
Let $M^B$ be the coarse moduli scheme over $\Q$
parameterizing the isomorphism classes of 
$(A,i)$, 
where 
\begin{itemize}
\item
$A$ is a two-dimensional abelian variety, and 
\item
$i:\cO\hookrightarrow\End(A)$
is an embedding of $\cO$ into the endomorphism ring of $A$.
\end{itemize}
Such a pair $(A,i)$ is called a \textit{QM-abelian surface} by $\cO$.
Then $M^B$ is a proper smooth curve over $\Q$, which is called the 
\textit{Shimura curve} associated to $B$.
Note that the $\Q$-isomorphism class of $M^B$ depends only on $d(B)$,
and does not depend on the choice of $\cO$.
We regard $M^B$ as an analog of a modular curve,
but $M^B$ has no cusps.
By \cite[Theorem 0]{Sh}, we have
$$M^B(\Q)=M^B(\R)=\emptyset.$$
If $k$ is an imaginary quadratic field of class number one
and if $B\otimes_{\Q}k\cong\M_2(k)$,
then $M^B(k)$ is never empty (see \cite[Proposition 6.5]{J}).

In the following, let 
\begin{itemize}
\item
$k$: an imaginary quadratic field of class number $h_k>1$.
\end{itemize}
We study the behavior of $M^B(k)$ when $d(B)$ grows.
Here, we regard $d(B)$ as the level of $M^B$.
The main result of this article is:
\begin{thm}
\label{mainthm}
There is a finite set $D(k)$ of positive integers depending on $k$ such that:
if $d(B)\not\in D(k)$, then $M^B(k)=\emptyset$.

\end{thm}

In practice, we prove:
\begin{thm}
\label{mainthm'}
\begin{enumerate}[\upshape (1)]
\item
There is a finite set $P(k)$ of prime numbers depending on $k$ such that:
if there is a prime divisor of $d(B)$ which is not in $P(k)$,
then $M^B(k)=\emptyset$.
\item
In the situation of {\rm (1)}, we can take $P(k)$ to satisfy the following:
There is an effectively computable constant $C(k)$ depending on $k$ such that
$$\text{$p\in P(k)$ implies $p<C(k)$}$$ with at most one possible exception.
\end{enumerate}
\end{thm}

In the situation of Theorem \ref{mainthm'},
let $\Dtil(k)$ be the set of the products of an even number ($\geq 2$)
of distinct prime numbers
in $P(k)$.
Then $\Dtil(k)$ is finite, and Theorem \ref{mainthm} holds with $D(k)=\Dtil(k)$.
Note that Theorems \ref{mainthm} and \ref{mainthm'}(1) are equivalent.
%

Jordan gave a similar result in \cite[Theorem 6.6]{J},
restricting the case where 
$B\otimes_{\Q}k\cong\M_2(k)$.
Note that no prime divisor of $d(B)$ splits in $k$ 
if $B\otimes_{\Q}k\cong\M_2(k)$.
Note also that a point of $M^B(k)$ is represented by a QM-abelian surface by $\cO$
over $k$ if and only if $B\otimes_{\Q}k\cong\M_2(k)$
(see \cite[Theorem 1.1]{J}).
These facts make the case 
$B\otimes_{\Q}k\cong\M_2(k)$
less difficult.

In \S \ref{excep_set}, we give a definition of the exceptional set $P(k)$. 
The proof of Theorem \ref{mainthm'} is given in 
\S \ref{sec:inert},\ref{sec:split},\ref{sec:ksplitsB}.
\S \ref{sec:isogeny} is devoted to its preparation.
The idea of the proof, which is a modification of that in \cite{J}, is as follows.
When $B\otimes_{\Q}k\not\cong\M_2(k)$, a point $x$ of $M^B(k)$ is not 
represented by a QM-abelian surface by $\cO$ over $k$.
We choose a quadratic extension $K$ of $k$ such that 
$B\otimes_{\Q}K\cong\M_2(K)$.
Then $x$ is represented by a QM-abelian surface by $\cO$ over $K$,
and we have a canonical isogeny character $\varrho$ of the absolute Galois 
group of $K$.
With the aid of the transfer map, we obtain a character $\varphi$ of the absolute 
Galois group of $k$.
Then $\varphi^{12}$ is independent of the choice of $K$ (up to $p$-th power),
and we classify it.
In each case of the classification, we carefully replace $K$ with a suitable one
if necessary, and compute $\varrho$.

%
Rational points on $M^B$ are studied in 
the context of the Hasse principle
in \cite{AMB1}, \cite{AMB2}, \cite{C}, \cite{RdV}, \cite{Sk}, \cite{SY}.
Especially, the result of \cite{AMB1} is applied to produce an explicit
infinite family of counterexamples to the Hasse principle on $M^B$ over number fields
$L$ when we fix $B$ and vary $L$ (see \cite[Proposition 2.6]{AMB2}).
The result of this article might be applied to produce such a family when we 
fix $L=k$ and vary $B$.

We also have a series of works \cite{Aeff}, \cite{AIV}, \cite{AII}, \cite{AIII}, \cite{AM}
concerning rational points on Shimura curves of $\Gamma_0(p)$-type
when the level $p$ grows.

\subsection*{Notation}

For a number field $K$, fix an embedding $K\hookrightarrow\C$.
We use the following notations, where $\mfQ$ is a prime of $K$.

\begin{itemize}
\item
$|a|$: the absolute value of $a\in\C$,
\item
$\overline{a}$: the complex conjugate of $a\in\C$,
\item
$\cO_K$: the ring of integers of $K$,
\item
$Cl_K$: the ideal class group of $K$,
\item
$\Kb$: the algebraic closure of $K$ inside $\C$,
\item
$\Kab$: the maximal abelian extension of $K$ inside $\Kb$,
\item
$\G_K:=\Gal(\Kb/K)$,
\item
$\Gab_K:=\Gal(\Kab/K)$,
\item
$\kappa(\mfQ)$: the residue field of $\mfQ$,
\item
$l_{\mfQ}$: the characteristic of $\kappa(\mfQ)$,
\item
$\N_{\mfQ}$: the cardinality of $\kappa(\mfQ)$,
\item
$\cO_{K,\mfQ}$: the completion of $\cO_K$ at $\mfQ$,
\item
$\util\in\kappa(\mfQ)$: the reduction of $u\in\cO_{K,\mfQ}$ modulo $\mfQ$.
\end{itemize}

\section{Exceptional sets}
\label{excep_set}

First, we define several sets of integers, which will be used to give a definition of
$P(k)$ in Theorem \ref{mainthm'}.
Let 
\begin{itemize}
\item
$\cS_0$: the set of non-principal primes of $k$ which split in $k/\Q$,
\item
$\cT$: the set of non-empty 
finite subsets of $\cS_0$ which generate $Cl_k$.
\end{itemize}
Then $\cS_0\ne\emptyset$ and $\cT\ne\emptyset$ 
since $h_k>1$.
For each prime $\mfq$ of $k$, 
fix a positive integer $h_{\mfq}$ such that $\mfq^{h_{\mfq}}$ is principal.
For example, we can take $h_{\mfq}$ to be the order of $\mfq$ in $Cl_k$;
we can also take $h_{\mfq}=h_k$.
Fix an element $\beta_{\mfq}\in\cO_k$
satisfying 
$$\mfq^{h_{\mfq}}=\bq\cO_k.$$
Then
$$\N_{\mfq}^{h_{\mfq}}=|\Norm_{k/\Q}(\bq)|=\Norm_{k/\Q}(\bq)=\bq\betab_{\mfq}
\quad\text{and}\quad |\bq|=\N_{\mfq}^{\frac{h_{\mfq}}{2}}.$$
We sometimes write 
$$h=h_{\mfq},\quad \beta=\bq$$ 
if there is no fear of confusion.
Let
\begin{itemize}
\item
$\cC_{\mfq}:=\\
\Set{\gamma^{24h_{\mfq}}+\gammab^{24h_{\mfq}}\in\Z
|\text{$\gamma\in\C$ is a root of
$X^2-mX+\N_{\mfq}=0$
for some $m\in\Z$, $m^2\leq 4\N_{\mfq}$}}$,
\item
$\cC'_{\mfq}:=\\
\Set{\gamma^{12h_{\mfq}}+\gammab^{12h_{\mfq}}\in\Z
|\text{$\gamma\in\C$ is a root of
$X^2-mX+\N_{\mfq}=0$
for some $m\in\Z$, $m^2\leq 4\N_{\mfq}$}}$,
\item
$\cA_{1,\mfq}:=\Set{a-\Tr_{k/\Q}(\bq^{24})\in\Z
|a\in\cC_{\mfq}}$,
\item
$\cA'_{1,\mfq}:=\Set{a-\Tr_{k/\Q}(\bq^{12})\in\Z
|a\in\cC'_{\mfq}}$,
\item
$\cA_{2,\mfq}:=\Set{a-\N_{\mfq}^{8h_{\mfq}}\Tr_{k/\Q}(\bq^8)\in\Z
|a\in\cC_{\mfq}}$,
\item
$\cA'_{2,\mfq}:=\Set{a-\N_{\mfq}^{4h_{\mfq}}\Tr_{k/\Q}(\bq^4)\in\Z
|a\in\cC'_{\mfq}}$,
\item
$\cA_{3,\mfq}:=\Set{a-2\N_{\mfq}^{12h_{\mfq}}\in\Z
|a\in\cC_{\mfq}}$,
\item
$\cA'_{3,\mfq}:=\Set{a-2\N_{\mfq}^{6h_{\mfq}}\in\Z
|a\in\cC'_{\mfq}}$,
\item
$\displaystyle\cA_{3,\cS}:=\bigcup_{\mfq\in\cS}\cA_{3,\mfq}$,
where $\cS\in\cT$,
\item
$\displaystyle\cA'_{3,\cS}:=\bigcup_{\mfq\in\cS}\cA'_{3,\mfq}$,
where $\cS\in\cT$.
%
\end{itemize}
These sets are finite.
Since $h_k>1$, we have $k\ne\Q(\sqrt{-1})$, $\Q(\sqrt{-3})$.
Then $\cO_k^{\times}=\Set{\pm 1}$, and so
$\cA_{1,\mfq}$, $\cA'_{1,\mfq}$, $\cA_{2,\mfq}$ and $\cA'_{2,\mfq}$
are independent of the choice of $\bq$.
%
%
%
%
Let
\begin{itemize}
\item
$\cP(\cA)$: the set of prime divisors of non-zero integers
in $\cA$,
where $\cA$ is a subset of $\Z$,
\item
$\cN$: the set of integers $N\in\Z$ such that
\begin{enumerate}[\upshape (i)]
\item
$N$ is the discriminant of a quadratic field, and
\item
for any prime number $2<l<\frac{|N|}{4}$, if $l$ splits in $k$,
then $l$ does not split in $\Q(\sqrt{N})$.
\end{enumerate}
\end{itemize}
Then $\cN$ is a finite set.
Furthermore, there is an effectively computable constant $C_0(k)$
depending on $k$ such that 
$$\text{$p\in\cN$ implies $p<C_0(k)$}$$
with at most one possible
exception
(see \cite[Theorem A]{Ma}).
Let
\begin{itemize}
\item
$\cN^{prime}$: the set of prime numbers in $\cN$,
\item
$\Ram$: the set of prime numbers which are ramified in $k$,
\item
$\cP_{\leq 23}:=\Set{2,3,5,7,11,13,17,19,23}$,
\item
$\cP_{\leq 7}:=\Set{2,3,5,7}$,
\item
$\displaystyle P_{ns}(k):=\Ram\cup\cP_{\leq 23}
\cup\left(\bigcap_{\mfq\in\cS_0}\cP(\cA_{1,\mfq})\right)
\cup\left(\bigcap_{\mfq\in\cS_0}\cP(\cA_{2,\mfq})\right)\\
\qquad\qquad \cup\left(\bigcap_{\cS\in\cT}
\Bigl(\cP(\cA_{3,\cS})\cup\Set{l_{\mfq}|\mfq\in\cS}\Bigr)\right)
\cup\cN^{prime}$,
\item
$\displaystyle P_{sp}(k):=\Ram\cup\cP_{\leq 7}
\cup\left(\bigcap_{\mfq\in\cS_0}\cP(\cA'_{1,\mfq})\right)
\cup\left(\bigcap_{\mfq\in\cS_0}\cP(\cA'_{2,\mfq})\right)\\
\qquad\qquad \cup\left(\bigcap_{\cS\in\cT}
\cP(\cA'_{3,\cS})\right)
\cup\cN^{prime}$.
\end{itemize}
These sets are finite.

Suppose $M^B(k)\ne\emptyset$,
and let $p$ be a prime divisor of $d(B)$.
%
%
%
%
%
In \S \ref{sec:inert},\ref{sec:split}, we will prove
\begin{equation}
\label{P_A(k)}
\text{$p\in P_{ns}(k)$\quad when $B\otimes_{\Q}k\not\cong\M_2(k)$.}
\end{equation}
In \S \ref{sec:ksplitsB}, 
we will prove
\begin{equation}
\label{P_A'(k)}
\text{$p\in P_{sp}(k)$\quad when $B\otimes_{\Q}k\cong\M_2(k)$.}
\end{equation}
Then Theorem \ref{mainthm'} holds with
$$P(k)=P_{ns}(k)\cup P_{sp}(k).$$
%

\section{Isogeny characters}
\label{sec:isogeny}

Fix 
\begin{itemize}
\item
a point $x\in M^B(k)$, and 
\item
a prime divisor $p$ of $d(B)$.
\end{itemize}
In the following, we study the characters associated to $x$ and $p$ (\cf\cite[\S 4]{J}).
Let $K_0$ be a quadratic extension of $k$.
Suppose that the following two equivalent conditions hold.
\begin{enumerate}[\upshape (C1)]
\item
$B\otimes_{\Q}K_0\cong\M_2(K_0)$.
\item
For any prime divisor $l$ of $d(B)$ splitting in $k$,
no prime of $k$ above $l$ splits in $K_0$.
\end{enumerate}
Note that we can always take such $K_0$
(see \cite[Remark 4.4]{AM}).
Let
\begin{equation*}
K:=
\begin{cases}
k&\text{if $B\otimes_{\Q}k\cong\M_2(k)$,}\\
K_0&\text{if $B\otimes_{\Q}k\not\cong\M_2(k)$.}
\end{cases}
\end{equation*}
Then $x$ is represented by a QM-abelian surface $(A,i)$ by $\cO$ over $K$
(see \cite[Theorem 1.1]{J}).
Fix such $(A,i)$.
Let $A[p]$ be the $p$-torsion subgroup of $A$. Then it is free of rank one over 
$\cO/p\cO$.
Note that we have an isomorphism
$$\cO/p\cO\cong
\Set{
\begin{pmatrix}a^p&b \\ 0&a\end{pmatrix}\in\M_2(\F_{p^2})}$$
of $\F_p$-algebras (see \cite[Chapitre II, Corollaire 1.7]{V}).
The module $A[p]$ has exactly one non-zero proper $\cO$-submodule, which we 
shall denote by $C_p$.
Let $\cP_{\cO}$ be the unique left ideal of $\cO$ of reduced norm $p\Z$;
it is in fact a two-sided ideal of $\cO$.
Then $C_p$ is free of rank one over $\cO/\cP_{\cO}$.
We have an isomorphism 
$\cO/\cP_{\cO}\cong\F_{p^2}$
of fields, which we fix.
The action of $\G_K$ on $C_p$ yields a character
$$\varrho:\G_K\longrightarrow\Aut_{\cO}(C_p)\cong\F_{p^2}^{\times},$$
where $\Aut_{\cO}(C_p)$ is the group of $\cO$-linear automorphisms of $C_p$.
Note that $\varrho$ depends on the choice of identification
$\cO/\cP_{\cO}\cong\F_{p^2}$,
but the unordered pair $\Set{\varrho, \varrho^p}$ does not depend on this choice.
Let 
$$\varrhoab:\Gab_K\longrightarrow\F_{p^2}^{\times}$$
be the character induced from $\varrho$.
Let
$$\varphi:\G_k\longrightarrow\F_{p^2}^{\times}$$
be the composition
\begin{equation*}
\label{phi}
\begin{CD}
\G_k@>\text{$\tr_{K/k}$}>>\Gab_K@>\text{$\varrhoab$}>>\F_{p^2}^{\times},
\end{CD}
\end{equation*}
where $\tr_{K/k}$
is the transfer map.
Note that we have $\varphi=\varrho$ when $K=k$.
Let
$$\Psi_{K/k}:Cl_k\longrightarrow Cl_K$$
be the map defined by $[\mfa]\longmapsto[\mfa\cO_K]$.

\begin{lem}
\label{phiindep}
Assume $B\otimes_{\Q}k\not\cong\M_2(k)$.
Then up to $p$-th power, the character $\varphi^4$ depends on the pair $(x,p)$,
and does not depend on the choice of $K_0$ or $(A,i)$.
\end{lem}

\begin{proof}
Since $h_k>1$, $x$ is not a CM point by $\Z[\sqrt{-1}]$ or
$\Z[\frac{-1+\sqrt{-3}}{2}]$
(see \cite[Theorem 5.12]{GR}).
Then the group of automorphisms of $(A,i)$ defined over $\kb$
is $\Set{\pm 1}$,
and the assertion follows from the same argument as in the proof of 
\cite[Lemma 5.3]{AM}.
\end{proof}

Fix a prime $\mfp$ of $k$ above $p$,
and a prime $\mfP$ of $K$ above $\mfp$.
For a later use, 
let $\betatil\in\kappa(\mfp)$ be the reduction of $\beta$ modulo $\mfp$.
Let
$$r(\mfP):\cO_{K,\mfP}^{\times}\longrightarrow\F_{p^2}^{\times}$$
$$\text{(\resp $s(\mfp):\cO_{k,\mfp}^{\times}\longrightarrow\F_{p^2}^{\times}$)}$$
be the composition
\begin{equation*}
\begin{CD}
\cO_{K,\mfP}^{\times}@>\text{$\omega$}>>
\Gab_K@>\text{$\varrhoab$}>>\F_{p^2}^{\times}
\end{CD}
\end{equation*}
\begin{equation*}
\begin{CD}
\text{(\resp}
\cO_{k,\mfp}^{\times}@>\text{$\omega$}>>
\Gab_k@>\text{$\varphiab$}>>\F_{p^2}^{\times}
\text{)},
\end{CD}
\end{equation*}
where
$\omega$ is the Artin map and
$\varphiab$ is the character induced from $\varphi$.
Note that the map
$\tr_{K/k}^{\text{ab}}:\Gab_k\longrightarrow\Gab_K$
induced from $\tr_{K/k}$ corresponds to the natural injection
$\A_k^{\times}\hookrightarrow\A_K^{\times}$
via class field theory, where $\A_k^{\times}$, $\A_K^{\times}$ are the groups of id\`{e}les
of $k$, $K$ respectively (see \cite[Chapter XIII, \S 9, Theorem 8]{We}).
Then
$$r(\mfP)(u)=s(\mfp)(u)$$
for any $u\in\cO_{k,\mfp}^{\times}$.
Let
$$\chi:\G_K\longrightarrow\F_p^{\times}$$
be the mod $p$ cyclotomic character.
Then by \cite[Proposotion 4.6]{J}, we have
$$\varrho^{p+1}=\chi.$$ 
Since $\sharp\F_p^{\times}=p-1$, we have 
$$\chi^{p-1}=\one,$$
where $\one$ is the trivial character.
We see that $\chi$ is unramified outside $p$.
By \cite[Proposition 4.7(2)]{J},
$\varrho^{12}$ is also unramified outside $p$.
Then $\chi, \varrho^{12}$ are identified with
characters
$$I_K(p)\longrightarrow\F_p^{\times},\quad
I_K(p)\longrightarrow\F_{p^2}^{\times}$$
respectively, where
$I_K(p)$ is the group of fractional ideals of $K$ prime to $p$.
%
Let 
$$\chi_{\mfP}:\cO_{K,\mfP}^{\times}\longrightarrow\F_p^{\times}$$
be the composition
\begin{equation*}
\begin{CD}
\cO_{K,\mfP}^{\times}@>\text{$\omega$}>>
\Gab_K@>\text{$\chiab$}>>\F_p^{\times},
\end{CD}
\end{equation*}
where
$\chiab$ is the character induced from $\chi$.
When $p>2$, let
$$\alpha:=\varrho^2\chi^{-\frac{p+1}{2}}.$$
Then
$$\alpha^6=\varrho^{12}\chi^{-3(p+1)}=\varrho^{12}\chi^{-6}.$$
Note that $\alpha^6$ is also identified with a character
$$I_K(p)\longrightarrow\F_{p^2}^{\times}.$$
%

Let $$\cO_p:=\cO\otimes_{\Z}\Z_p, \quad B_p:=B\otimes_{\Q}\Q_p.$$
The action of $\G_K$ on the $p$-adic Tate module $T_pA$
yields a representation
$$R:\G_K\longrightarrow\Aut_{\cO}(T_pA)\cong\cO_p^{\times}\subseteq B_p^{\times},$$
where $\Aut_{\cO}(T_pA)$ is the group of $\Z_p$-linear automorphisms of $T_pA$
commuting with the action of $\cO$.
%
For a prime $\mfQ$ of $K$, let
$F_{\mfQ}\in\G_K$ be a Frobenius element at $\mfQ$.
%
For any $t\in B_p$, let $t^{\iota}$ be the conjugate of $t$, and let
$$\Nrd_{B_p/\Q_p}(X-t):=(X-t)(X-t^{\iota})\in\Q_p[X].$$
Assume $l_{\mfQ}\ne p$.
Then for any positive integer $n$, there is an integer $a(F_{\mfQ}^n)\in\Z$ satisfying
$$\Nrd_{B_p/\Q_p}(X-R(F_{\mfQ}^n))=X^2-a(F_{\mfQ}^n)X+\N_{\mfQ}^n\in\Z[X].$$
Let $\pi=\pi_{\mfQ}\in\C$ be a root of
$X^2-a(F_{\mfQ})X+\N_{\mfQ}=0$.
Then the roots of this equation are $\pi,\pib$.
We have
$$\pi\pib=\N_{\mfQ},\quad
|\pi|=\sqrt{\N_{\mfQ}},\quad
\pi^n+\pib^n=a(F_{\mfQ}^n)\quad
\text{and}\quad
a(F_{\mfQ}^n)^2\leq 4\N_{\mfQ}^n$$
for any positive integer $n$ (see \cite[Proposition 5.3]{J}).
%
If $\N_{\mfQ}$ is an odd power of $l_{\mfQ}$,
then $\Q(\pi)$ is an imaginary quadratic field.
By \cite[Proposition 5.2]{J}, we have
$$\Nrd_{B_p/\Q_p}(X-R(F_{\mfQ}))\equiv
(X-\varrho(F_{\mfQ}))(X-\varrho(F_{\mfQ})^p)\bmod{p}.$$
Fix a prime $\mfp_0$ of $\Q(\pi)$ above $p$.
Then
$$(X-\pi)(X-\pib)=X^2-a(F_{\mfQ})X+\N_{\mfQ}=\Nrd_{B_p/\Q_p}(X-R(F_{\mfQ}))$$
$$\equiv (X-\varrho(F_{\mfQ}))(X-\varrho(F_{\mfQ})^p)\bmod{\mfp_0}.$$
By replacing $\pi$ with $\pib$ if necessary, we may assume
\begin{equation}
\label{rhopip}
\varrho(F_{\mfQ})\equiv\pi\bmod{\mfp_0}
\quad \text{and}\quad
\varrho^p(F_{\mfQ})\equiv\pib\bmod{\mfp_0}.
\end{equation}

%

\section{The case where $B\otimes_{\Q}k\not\cong\M_2(k)$ and $p$ is inert in $k$}
\label{sec:inert}

In this section, we assume 
\begin{itemize}
\item
$B\otimes_{\Q}k\not\cong\M_2(k)$, and
\item
$p$ is inert in $k$.
\end{itemize}
Then $K=K_0$, $\mfp=p\cO_k$ and $\mfp\not\in\cS_0$.
First, we classify the characters $r(\mfP)$, $s(\mfp)$.

\begin{prop}
\label{cbaai}
Suppose $p>7$.
Then:
\begin{enumerate}[\upshape (1)]
\item
If $\mfP=p\cO_K$, then
there is an integer
$b\in\Set{12,12p,4(p+2),4(2p+1),6(p+1)}$
such that
$$r(\mfP)^{12}(u)=\Norm_{\F_{p^4}/\F_{p^2}}(\util)^{-b}$$
for any
$u\in\cO_{K,\mfP}^{\times}$.
Furthermore, we can take
\begin{equation*}
\begin{cases}
b\in\Set{12,12p,6(p+1)} &\text{if $p\not\equiv 1\bmod{3}$, and}\\
b\in\Set{12,12p,4(p+2),4(2p+1)} &\text{if $p\not\equiv 1\bmod{4}$.}
\end{cases}
\end{equation*}
\item
There is an integer $c\in\Set{24,24p,8(p+2),8(2p+1),12(p+1)}$
such that
$$s(\mfp)^{12}(u)=\util^{-c}$$
for any $u\in\cO_{k,\mfp}^{\times}$.
Furthermore, we can take
\begin{equation*}
\begin{cases}
c\in\Set{24,24p,12(p+1)} &\text{if $p\not\equiv 1\bmod{3}$, and}\\
c\in\Set{24,24p,8(p+2),8(2p+1)} &\text{if $p\not\equiv 1\bmod{4}$.}
\end{cases}
\end{equation*}
\end{enumerate}
\end{prop}

\begin{proof}

(1)
It follows from \cite[Corollary 4.13 and Remarks 4.10, 4.14]{J}.
Note that the case $(a,k,h)=(4,0,2)$ is missing in the list of
\cite[Remark 4.14]{J}.

(2)
First, note that
$s(\mfp)^{12}(u)=\util^{-c}$
for some $c\in\Set{24,24p,8(p+2),8(2p+1),12(p+1)}$
if and only if
$s(\mfp)^{12p}(u)=\util^{-c'}$
for some $c'\in\Set{24,24p,8(p+2),8(2p+1),12(p+1)}$,
because
$(\util^{-24})^p=\util^{-24p}$
and
$(\util^{-24p})^p=\util^{-24p^2}=\util^{-24}\in\F_{p^2}^{\times}$, etc.
Then the assertion does not change even if we replace $\varphi^4$ (or $\varphi^{12}$)
with its $p$-th power.
By replacing $K_0$ if necessary, we may assume that $\mfp$
is inert in $K$ (see (C2), Lemma \ref{phiindep} and \cite[Remark 4.4]{AM}).
Then $\mfP=p\cO_K$.
Since
$s(\mfp)^{12}(u)=r(\mfP)^{12}(u)$
for any $u\in\cO_{k,\mfp}^{\times}$,
the assertion follows from (1).
\end{proof}



Note that
$24<8(p+2)<12(p+1)<8(2p+1)<24p$.
Suppose that we are in the situation of Proposition \ref{cbaai}(2).
We take $c$ to satisfy
\begin{equation*}
\begin{cases}
c\in\Set{24,24p,12(p+1)} &\text{if $p\not\equiv 1\bmod{3}$, and}\\
c\in\Set{24,24p,8(p+2),8(2p+1)} &\text{if $p\not\equiv 1\bmod{4}$.}
\end{cases}
\end{equation*}
In the following, we study the cases
[$c=24$ or $24p$], [$c=8(p+2)$ or $8(2p+1)$] and [$c=12(p+1)$]
separately.

[Case $c=24$ or $24p$].

\begin{prop}
\label{c=24}
We have
$\displaystyle p\in\cP_{\leq 23}\cup\bigcap_{\mfq\in\cS_0}\cP(\cA_{1,\mfq})$.

\end{prop}

\begin{proof}

Suppose $p>23$.
We prove 
$\displaystyle p\in\bigcap_{\mfq\in\cS_0}\cP(\cA_{1,\mfq})$
as follows.

Fix any $\mfq\in\cS_0$.
Then $\mfp\ne \mfq$.
By replacing $K_0$ if necessary, we may assume that
$\mfp$ is inert in $K$
and that
$\mfq$ is ramified in $K$.
Note that $c=24$ and $c=24p$ might interchange in this process
(see Lemma \ref{phiindep}).
By replacing $\varrho$ with $\varrho^p$ if necessary, we may assume $c=24$.
Then
$s(\mfp)^{12}(u)=\util^{-24}$
for any $u\in\cO_{k,\mfp}^{\times}$.

First, we prove
$$r(\mfP)^{12}(u)=\Norm_{\F_{p^4}/\F_{p^2}}(\util)^{-12}$$
for any
$u\in\cO_{K,\mfP}^{\times}$.
We have $\mfP=\mfp\cO_K=p\cO_K$.
By Proposition \ref{cbaai}(1), there is an integer
$b\in\Set{12,12p,4(p+2),4(2p+1),6(p+1)}$
such that 
$r(\mfP)^{12}(u)=\Norm_{\F_{p^4}/\F_{p^2}}(\util)^{-b}$
for any
$u\in\cO_{K,\mfP}^{\times}$.
In this case,
$\util^{-24}=s(\mfp)^{12}(u)=r(\mfP)^{12}(u)=\util^{-2b}$
for any $u\in\cO_{k,\mfp}^{\times}$.
Choose an element $w\in\cO_{k,\mfp}^{\times}$ such that
$\wtil$ is a generator of the cyclic group 
$\kappa(\mfp)^{\times}=\F_{p^2}^{\times}$ $(\cong\Z/(p^2-1)\Z)$.
Then $\wtil^{-24}=\wtil^{-2b}$, and so $2b\equiv 24\bmod{(p^2-1)}$.
%
%
Since $p>23$, we have $b=12$.
Therefore
$r(\mfP)^{12}(u)=\Norm_{\F_{p^4}/\F_{p^2}}(\util)^{-12}$
for any
$u\in\cO_{K,\mfP}^{\times}$.

Let $q=l_{\mfq}$, and let $\mfQ$ be the unique prime of $K$ above $\mfq$.
Then $\N_{\mfQ}=\N_{\mfq}=q$ and $\mfq\cO_K=\mfQ^2$.
Note that $p\ne q$.
We claim
\begin{enumerate}[\upshape (i)]
\item
$a(F_{\mfQ}^{24h_{\mfq}})\equiv\Tr_{k/\Q}(\bq^{24})\bmod{p}$,
and
\item
$a(F_{\mfQ}^{24h_{\mfq}})\ne\Tr_{k/\Q}(\bq^{24})$,
\end{enumerate}
which are proved as follows.

For simplicity, we write $h=h_{\mfq}$, $\beta=\bq$.

(i)
We have
\begin{equation*}
\varrho(F_{\mfQ}^{24h})=\varrho^{12}(F_{\mfQ}^{2h})
=\varrho^{12}(\mfQ^{2h})=\varrho^{12}(\mfq^h\cO_K)
=\varrho^{12}(\beta\cO_K)
\end{equation*}
\begin{equation}
\label{rhoF24h}
=\varrho^{12}((1)_{\infty},(1)_p,(\beta)^{\infty,p})
=\varrho^{12}((\beta^{-1})_{\infty},(\beta^{-1})_p,(1)^{\infty,p})
\end{equation}
\begin{equation*}
=\varrho^{12}((\beta^{-1})_p,(1)^p)
=r(\mfP)^{12}(\beta^{-1}).
\end{equation*}
%
Here, 
\begin{itemize}
\item
$\infty$ is the infinite place of $\Q$,
\item
$((1)_{\infty},(1)_p,(\beta)^{\infty,p})$ is the id\`{e}le of $K$ where the components 
above $\infty$, $p$ are $1$ and the others are $\beta$,
\item
$((\beta^{-1})_{\infty},(\beta^{-1})_p,(1)^{\infty,p})$
is the id\`{e}le of $K$ where the components above $\infty$, $p$ are $\beta^{-1}$
and the others are $1$,
\item
$((\beta^{-1})_p,(1)^p)$
is the id\`{e}le of $K$ where the component above $p$ is $\beta^{-1}$
and the others are $1$.
\end{itemize}
%
Note that the components above $\infty$ have no contribution since $K$ has
no real place.
Then
$\varrho(F_{\mfQ}^{24h})=\Norm_{\F_{p^4}/\F_{p^2}}(\betatil^{-1})^{-12}=\betatil^{24}$,
where $\betatil\in\kappa(\mfp)=\F_{p^2}$
is the reduction of $\beta$ modulo $\mfp$.
By \cite[Proposition 5.3]{J}, we obtain
$$a(F_{\mfQ}^{24h})
\equiv\Tr_{\F_{p^2}/\F_p}(\varrho(F_{\mfQ}^{24h}))
=\Tr_{\F_{p^2}/\F_p}(\betatil^{24})
\equiv\Tr_{k/\Q}(\beta^{24})\bmod{p}.$$

(ii)
Suppose $a(F_{\mfQ}^{24h})=\Tr_{k/\Q}(\beta^{24})$.
Since $\beta\betab=\N_{\mfq}^h=q^h$, the roots of
\begin{equation}
\label{eq1}
X^2-\Tr_{k/\Q}(\beta^{24})X+q^{24h}=0
\end{equation}
are $\beta^{24},\betab^{24}$.
On the other hand,
the roots of
$X^2-a(F_{\mfQ})X+q=0$
are $\pi,\pib$.
We have $\pi\pib=q$ and
$a(F_{\mfQ}^{24h})=\pi^{24h}+\pib^{24h}$
(see the last paragraph of \S \ref{sec:isogeny}).
Then $\pi^{24h},\pib^{24h}$ are the roots of (\ref{eq1})
since $a(F_{\mfQ}^{24h})=\Tr_{k/\Q}(\beta^{24})$.
Hence we have an equality of unordered pairs
$\Set{\beta^{24},\betab^{24}}=\Set{\pi^{24h},\pib^{24h}}$.
Then
$$\Q(\pi)\supseteq\Q(\pi^{24h})=\Q(\beta^{24})\subseteq\Q(\beta)\subseteq k.$$
We prove $\beta^{24}\not\in\Q$.
Assume otherwise, i.e., $\beta^{24}\in\Q$.
Then
$q^{24h}
=\beta^{24}\betab^{24}=(\beta^{24})^2$, and so
$\beta^{24}=\pm q^{12h}$.
Hence
$\mfq^{24h}=\beta^{24}\cO_k=q^{12h}\cO_k$ and
$\mfq^2=q\cO_k$.
This implies that $\mfq$ is ramified in $k/\Q$,
which contradicts $\mfq\in\cS_0$.
Therefore $\beta^{24}\not\in\Q$.
Since $[\Q(\pi):\Q]=[k:\Q]=2$,
we have
$$\Q(\pi)=\Q(\pi^{24h})=\Q(\beta^{24})=\Q(\beta)=k.$$
Then $\pi\in k$.
Since $\pi\pib=q$, we have
$\mfq=\pi\cO_k$ or $\pib\cO_k$.
This implies that $\mfq$ is principal, which contradicts $\mfq\in\cS_0$.
Then (ii) has been proved.

By (i) and (ii), $p$ divides the non-zero integer
$a(F_{\mfQ}^{24h})-\Tr_{k/\Q}(\beta^{24})$.
Since $a(F_{\mfQ}^{24h})=\pi^{24h}+\pib^{24h}\in\cC_{\mfq}$, we have
$a(F_{\mfQ}^{24h})-\Tr_{k/\Q}(\beta^{24})\in\cA_{1,\mfq}$.
Then $p\in\cP(\cA_{1,\mfq})$.
Therefore $\displaystyle p\in\bigcap_{\mfq\in\cS_0}\cP(\cA_{1,\mfq})$.
\end{proof}

[Case $c=8(p+2)$ or $8(2p+1)$].
In this case, we have $p\equiv 1\bmod{3}$.

\begin{prop}
\label{c=8(p+2)}
We have
$\displaystyle p\in\bigcap_{\mfq\in\cS_0}\cP(\cA_{2,\mfq})$.

\end{prop}

\begin{proof}

We repeat the argument in the proof of Proposition \ref{c=24}.
Fix any $\mfq\in\cS_0$.
We may assume that
$\mfp$ is inert in $K$
and that
$\mfq$ is ramified in $K$.
Then $\mfP=\mfp\cO_K=p\cO_K$. 
By replacing $\varrho$ with $\varrho^p$ if necessary, we may assume $c=8(p+2)$.

We prove
$$r(\mfP)^{12}(u)=\Norm_{\F_{p^4}/\F_{p^2}}(\util)^{-4(p+2)}$$
for any
$u\in\cO_{K,\mfP}^{\times}$.
There is an integer
$b\in\Set{12,12p,4(p+2),4(2p+1),6(p+1)}$
such that
$r(\mfP)^{12}(u)=\Norm_{\F_{p^4}/\F_{p^2}}(\util)^{-b}$
for any
$u\in\cO_{K,\mfP}^{\times}$.
In this case, 
$\util^{-8(p+2)}=s(\mfp)^{12}(u)=r(\mfP)^{12}(u)=\util^{-2b}$
for any $u\in\cO_{k,\mfp}^{\times}$.
Then $2b\equiv 8(p+2)\bmod{(p^2-1)}$.
%
Since $p>7$, we conclude $b=4(p+2)$.

Let $q=l_{\mfq}$, and
let $\mfQ$ be the prime of $K$ above $\mfq$.
We claim
\begin{enumerate}[\upshape (i)]
\item
$a(F_{\mfQ}^{24h_{\mfq}})\equiv q^{8h_{\mfq}}\Tr_{k/\Q}(\bq^8)\bmod{p}$,
and
\item
$a(F_{\mfQ}^{24h_{\mfq}})\ne q^{8h_{\mfq}}\Tr_{k/\Q}(\bq^8)$.
\end{enumerate}


(i)
By (\ref{rhoF24h}), we have 
$$\varrho(F_{\mfQ}^{24h})=r(\mfP)^{12}(\beta^{-1})
=\Norm_{\F_{p^4}/\F_{p^2}}(\betatil^{-1})^{-4(p+2)}
=\betatil^{8(p+2)}.$$
Since $\betab^p\equiv\beta\bmod{\mfp}$,
$\beta^p\equiv\betab\bmod{\mfp}$
and
$\beta\betab=q^h$,
we have
$$a(F_{\mfQ}^{24h})
\equiv\Tr_{\F_{p^2}/\F_p}(\varrho(F_{\mfQ}^{24h}))
=\Tr_{\F_{p^2}/\F_p}(\betatil^{8(p+2)})
\equiv\Tr_{k/\Q}(\beta^{8(p+2)})$$
$$=\beta^{8p+16}+\betab^{8p+16}
\equiv\betab^8\beta^{16}+\beta^8\betab^{16}
=\beta^8\betab^8(\beta^8+\betab^8)
=q^{8h}\Tr_{k/\Q}(\beta^8)\bmod{p}.$$

(ii)
Suppose $a(F_{\mfQ}^{24h})=q^{8h}\Tr_{k/\Q}(\beta^8)$.
Since $\beta^8\betab^8=q^{8h}$, the roots of
\begin{equation}
\label{eq2}
X^2-\Tr_{k/\Q}(\beta^8)X+q^{8h}=0
\end{equation}
are $\beta^8,\betab^8$.
On the other hand,
the roots of
$X^2-a(F_{\mfQ})X+q=0$
are $\pi,\pib$.
Then
$\pi^{24h}+\pib^{24h}=a(F_{\mfQ}^{24h})=q^{8h}\Tr_{k/\Q}(\beta^8)$
and
$\pi\pib=q$.
Hence
$\frac{\pi^{24h}}{q^{8h}}+\frac{\pib^{24h}}{q^{8h}}=\Tr_{k/\Q}(\beta^8)$
and
$\frac{\pi^{24h}}{q^{8h}}\cdot\frac{\pib^{24h}}{q^{8h}}=q^{8h}$.
Then
$\frac{\pi^{24h}}{q^{8h}},\frac{\pib^{24h}}{q^{8h}}$
are the roots of (\ref{eq2}).
Hence
$\Set{\beta^8,\betab^8}=\Set{\frac{\pi^{24h}}{q^{8h}},\frac{\pib^{24h}}{q^{8h}}}$,
and so
$$\Q(\pi)\supseteq\Q(\pi^{24h})=\Q(\beta^8)\subseteq\Q(\beta)\subseteq k.$$
%
Assume $\beta^8\in\Q$.
Then $q^{8h}=\beta^8\betab^8=(\beta^8)^2$, and so $\beta^8=\pm q^{4h}$.
Hence $\mfq^{8h}=\beta^8\cO_k=q^{4h}\cO_k$
and
$\mfq^2=q\cO_k$.
This contradicts $\mfq\in\cS_0$.
Therefore $\beta^8\not\in\Q$.
Then
$$\Q(\pi)=\Q(\pi^{24h})=\Q(\beta^8)=\Q(\beta)=k,$$
and so $\pi\in k$.
Since $\pi\pib=q$, we have
$\mfq=\pi\cO_k$ or $\pib\cO_k$.
This contradicts $\mfq\in\cS_0$.
Then (ii) has been proved.

Since $a(F_{\mfQ}^{24h})-q^{8h}\Tr_{k/\Q}(\beta^8)\in\cA_{2,\mfq}$,
we conclude $p\in\cP(\cA_{2,\mfq})$.
Note that we do not use the congruence $p\equiv 1\bmod{3}$.
\end{proof}

[Case $c=12(p+1)$].
In this case, we have $p\equiv 1\bmod{4}$.

\begin{prop}
\label{c=12(p+1)}
We have
$\displaystyle p\in\left(\bigcap_{\cS\in\cT}
\cP(\cA_{3,\cS})\right)
\cup\cN^{prime}$.
\end{prop}

Fix any $\cS\in\cT$.
Then $\cS$ is a non-empty finite subset of $\cS_0$ which generates $Cl_k$.
First, we prove:

\begin{lem}
\label{rho12chi-6}
Suppose that $\mfp$ is inert in $K$ and that
any prime in $\cS$ is ramified in $K$. Then:
\begin{enumerate}[\upshape (1)]
\item
The character
$\varrho^{12}\chi^{-6}:\G_K\longrightarrow\F_{p^2}^{\times}$
is unramified everywhere.
\item
If $p\not\in\cP(\cA_{3,\cS})$, then the character
$Cl_K\longrightarrow\F_{p^2}^{\times}$
induced from $\varrho^{12}\chi^{-6}$ is trivial on $\Psi_{K/k}(Cl_k)$.
\end{enumerate}
\end{lem}

\begin{proof}

(1)
Since $\varrho^{12}\chi^{-6}$ is unramified outside $p$, we have only to
show that it is unramified at $\mfP$ ($=\mfp\cO_K=p\cO_K$).
Since $p>7$ and $p\equiv 1\bmod{4}$, we have $p>11$.
There is an integer
$b\in\Set{12,12p,4(p+2),4(2p+1),6(p+1)}$
such that
$r(\mfP)^{12}(u)=\Norm_{\F_{p^4}/\F_{p^2}}(\util)^{-b}$
for any
$u\in\cO_{K,\mfP}^{\times}$.
In this case, $\util^{-12(p+1)}=s(\mfp)^{12}(u)=r(\mfP)^{12}(u)=\util^{-2b}$
for any $u\in\cO_{k,\mfp}^{\times}$.
Then $2b\equiv 12(p+1)\bmod{(p^2-1)}$.
%
Since $p>11$, we have $b=6(p+1)$.
Then 
$$r(\mfP)^{12}(u)=\Norm_{\F_{p^4}/\F_{p^2}}(\util)^{-6(p+1)}
=(\Norm_{\F_{p^4}/\F_{p^2}}(\util)^{p+1})^{-6}$$
$$=(\Norm_{\F_{p^2}/\F_p}\circ\Norm_{\F_{p^4}/\F_{p^2}}(\util))^{-6}
=\Norm_{\F_{p^4}/\F_p}(\util)^{-6}$$
and
$\chi_{\mfP}(u)=\Norm_{\F_{p^4}/\F_p}(\util)^{-1}$
for any
$u\in\cO_{K,\mfP}^{\times}$
(\cf \cite[Proof of Proposition 4.8]{J}).
Hence
$r(\mfP)^{12}(u)\chi_{\mfP}^{-6}(u)=1$
for any $u\in\cO_{K,\mfP}^{\times}$.
Therefore $\rho^{12}\chi^{-6}$ is unramified at $\mfP$.

(2)
Assume $p\not\in\cP(\cA_{3,\cS})$.
Fix any $\mfq\in\cS$, let $q=l_{\mfq}$,
and let $\mfQ$ be the prime of $K$ above $\mfq$.
Then $\N_{\mfQ}=\N_{\mfq}=q$ and $p\ne q$.
By (\ref{rhoF24h}), we have
$$\varrho(F_{\mfQ}^{24h})=r(\mfP)^{12}(\beta^{-1})
=\Norm_{\F_{p^4}/\F_{p^2}}(\betatil^{-1})^{-6(p+1)}
=\betatil^{12(p+1)},$$
and so
$$a(F_{\mfQ}^{24h})
\equiv\Tr_{\F_{p^2}/\F_p}(\varrho(F_{\mfQ}^{24h}))
=\Tr_{\F_{p^2}/\F_p}(\betatil^{12(p+1)})
=\betatil^{12(p+1)}+\betatil^{12(p^2+p)}$$
$$=2\betatil^{12(p+1)}
\equiv 2\beta^{12}\betab^{12}
=2q^{12h}\bmod{p}.$$
Here, note that
$\betatil^{p^2}=\betatil\in\kappa(\mfp)=\F_{p^2}$.
Suppose $a(F_{\mfQ}^{24h})\ne 2q^{12h}$.
Then $p$ divides the non-zero integer
$a(F_{\mfQ}^{24h})-2q^{12h}\in\cA_{3,\mfq}\subseteq\cA_{3,\cS}$,
and so
$p\in\cP(\cA_{3,\cS})$.
This contradicts the assumption $p\not\in\cP(\cA_{3,\cS})$.
Hence $a(F_{\mfQ}^{24h})=2q^{12h}$.
Then
$\pi^{24h}+\pib^{24h}=a(F_{\mfQ}^{24h})=2q^{12h}$
and
$\pi^{24h}\pib^{24h}=q^{24h}$.
Therefore
$\pi^{24h}=\pib^{24h}=q^{12h}$.

We prove 
$\pi^{24}=q^{12}$
(\cf \cite[Proof of Theorem 5.1]{Aeff}).
%
Let $\zeta:=\pib\pi^{-1}$.
Then $\zeta^{24h}=\pib^{24h}\pi^{-24h}=1$
and $\pib=\zeta\pi$.
Hence 
$$\Q(\pi)=\Q(\pib)=\Q(\zeta\pi)=\Q(\pi, \zeta)\supseteq\Q(\zeta).$$
Since $[\Q(\pi):\Q]=2$,
we have $\zeta^4=1$ or $\zeta^6=1$.
Then $\zeta^{12}=1$.
This implies $\pib^{12}=\zeta^{12}\pi^{12}=\pi^{12}$,
and so $\pi^{12}\in\Q$.
Here, note that $\Q(\pi)$ is an imaginary quadratic field.
Since $|\pi|=\sqrt{q}$, we have $|\pi^{12}|=q^6$.
Therefore $\pi^{12}=\pm q^6$ and $\pi^{24}=q^{12}$.

By (\ref{rhopip}), we have
$\varrho(F_{\mfQ})\equiv\pi\bmod{\mfp_0}$.
Then
$$\varrho^{12}(\mfq\cO_K)=\varrho^{12}(\mfQ^2)=\varrho^{24}(F_{\mfQ})
\equiv\pi^{24}=q^{12}\equiv\chi(F_{\mfQ})^{12}$$
$$=\chi(\mfQ)^{12}=\chi^6(\mfQ^2)=\chi^6(\mfq\cO_K)\bmod{p}.$$
Since $\cS$ generates $Cl_k$, the assertion follows.
\end{proof}

\noindent
\textit{Proof of Proposition \ref{c=12(p+1)}}.
\ \ 
Assume
$p\not\in\displaystyle \bigcap_{\cS\in\cT}\cP(\cA_{3,\cS})$.
Fix any prime number $q<\frac{p}{4}$ that is not inert in $k$.
Then we prove $\left(\frac{q}{p}\right)=-1$ as follows,
where $\left(\frac{q}{p}\right)$ is the Legendre symbol.

There is an element $\cS\in\cT$ such that $p\not\in\cP(\cA_{3,\cS})$.
Fix such $\cS$.
Let $\mfq$ be a prime of $k$ above $q$.
Since $p$ is inert in $k$ and $p>4q>q$, we have
$\mfp\not\in\cS\cup\Set{\mfq}$.
By replacing $K_0$ if necessary,
we may assume that $\mfp$ is inert in $K$ and that
any prime in $\cS\cup\Set{\mfq}$ is ramified in $K$.
Here, note that $c=12(p+1)$ does not change even if $\varphi^{12}$ is replaced with
$\varphi^{12p}$.
Let $\mfQ$ be the prime of $K$ above $\mfq$.
Since $q$ is not inert in $k$, we have
$\N_{\mfQ}=\N_{\mfq}=q$.
Recall
$\alpha=\varrho^2\chi^{-\frac{p+1}{2}}$
and
$\alpha^6=\varrho^{12}\chi^{-6}$.
By Lemma \ref{rho12chi-6}(2), we have
$$\alpha(F_{\mfQ})^{12}
=\alpha^{12}(\mfQ)
=\alpha^6(\mfq\cO_K)
=\varrho^{12}(\mfq\cO_K)\chi^{-6}(\mfq\cO_K)
=1.$$
Since $\varrho^{p+1}=\chi$, we have
$$\alpha^{\frac{p+1}{2}}=\varrho^{p+1}\chi^{-\frac{(p+1)^2}{4}}
=\chi^{1-\frac{(p+1)^2}{4}}
=\chi^{\frac{(1-p)(3+p)}{4}}.$$
Recall $p\equiv 1\bmod{4}$.
Then $\frac{3+p}{4}\in\Z$, and so
$\alpha^{\frac{p+1}{2}}=\one$.
We also see that
$\frac{p+1}{2}$ is odd.
Since the order of $\alpha(F_{\mfQ})$ divides $12$ and $\frac{p+1}{2}$, we have
$\alpha(F_{\mfQ})^3=1$.
Therefore
$$\text{$\alpha(F_{\mfQ})+\alpha(F_{\mfQ})^{-1}=-1$ or $2$.}$$
By (\ref{rhopip}), we have
$\pi^2+\pib^2\equiv\varrho(F_{\mfQ})^2+\varrho^p(F_{\mfQ})^2\bmod{\mfp_0}$.
Since $\varrho^p=\chi\varrho^{-1}$
and $\left(\frac{q}{p}\right)\equiv q^{\frac{p-1}{2}}\bmod{p}$,
we have
$$\pi^2+\pib^2
\equiv\varrho(F_{\mfQ})^2+\chi(F_{\mfQ})^2\varrho(F_{\mfQ})^{-2}
=\chi(F_{\mfQ})^{\frac{p+1}{2}}\alpha(F_{\mfQ})
+\chi(F_{\mfQ})^{2-\frac{p+1}{2}}\alpha(F_{\mfQ})^{-1}$$
$$=\chi(F_{\mfQ})^{\frac{p+1}{2}}(\alpha(F_{\mfQ})+\alpha(F_{\mfQ})^{-1})
=q^{\frac{p+1}{2}}(\alpha(F_{\mfQ})+\alpha(F_{\mfQ})^{-1})$$
$$=q\left(\frac{q}{p}\right)(\alpha(F_{\mfQ})+\alpha(F_{\mfQ})^{-1})\bmod{p}.$$
Assume
$\left(\frac{q}{p}\right)=1$.
Then
$\pi^2+\pib^2\equiv -q$ or $2q\bmod{p}$.
Since $\pi\pib=q$, we have
$(\pi+\pib)^2\equiv q$ or $4q\bmod{p}$.
We observe that
$0\leq (\pi+\pib)^2\leq 4q$.
Then
$(\pi+\pib)^2= q$ or $4q$
because
$|(\pi+\pib)^2-q|\leq 3q<p$
and
$|(\pi+\pib)^2-4q|\leq 4q<p$.
This contradicts $\pi+\pib\in\Z$.
Therefore $\left(\frac{q}{p}\right)=-1$.

Since $p\equiv 1\bmod{4}$, $p$ is the discriminant of $\Q(\sqrt{p})$.
Let $q$ be any prime number such that $2<q<\frac{p}{4}$.
Assume that $q$ splits in $k$.
Then we have proved $\left(\frac{q}{p}\right)=-1$.
Since $p\equiv 1\bmod{4}$ again, we have
$\left(\frac{p}{q}\right)=\left(\frac{q}{p}\right)(-1)^{\frac{(p-1)(q-1)}{4}}
=\left(\frac{q}{p}\right)=-1$.
Then $q$ is inert (and so does not split) in $\Q(\sqrt{p})$.
Therefore $p\in\cN^{prime}$.
%
\qed

\section{The case where $B\otimes_{\Q}k\not\cong\M_2(k)$ and $p$ splits in $k$}
\label{sec:split}

In this section, we assume 
\begin{itemize}
\item
$B\otimes_{\Q}k\not\cong\M_2(k)$, and
\item
$p$ splits in $k$.
\end{itemize}
Then $K=K_0$.
We have the following restriction on $p$.

\begin{lem}
\label{p=2or}
We have $p=2$ or $p\equiv 1\bmod{4}$.

\end{lem}

\begin{proof}

We can regard $k$ as a subfield of $\Q_p$ since $p$ splits in $k$.
Since $M^B(k)\ne\emptyset$, we have $M^B(\Q_p)\ne\emptyset$.
Then by \cite[Theorem 5.6]{JL}, we have either
\begin{enumerate}[\upshape (i)]
\item
$p=2$ and $d(B)=2q_1\cdots q_{2r-1}$ with the prime numbers $q_i$, $1\leq i\leq 2r-1$,
satisfying $q_i\equiv 3\bmod{4}$, or
\item
$d(B)=2p$ with $p\equiv 1\bmod{4}$.
\end{enumerate}
Then the assertion follows.
\end{proof}

Our goal of this section is:

\begin{prop}
\label{psplit}
We have
$\displaystyle p\in\Set{2}\cup\left(\bigcap_{\cS\in\cT}
\Bigl(\cP(\cA_{3,\cS})\cup\Set{l_{\mfq}|\mfq\in\cS}\Bigr)\right)
\cup\cN^{prime}$.

\end{prop}

Let $\mfp_1,\mfp_2$ be the primes of $k$ above $p$. Fix any $\cS\in\cT$.
First, we prove:

\begin{lem}
\label{rho12chi-6'}
Suppose that any prime in $\cS\cup\Set{\mfp_1,\mfp_2}$ is ramified in $K$.
Then:
\begin{enumerate}[\upshape (1)]
\item
The character
$\varrho^{12}\chi^{-6}:\G_K\longrightarrow\F_{p^2}^{\times}$
is unramified everywhere.
\item
If $p\not\in\cP(\cA_{3,\cS})\cup\Set{l_{\mfq}|\mfq\in\cS}$,
then the character
$Cl_K\longrightarrow\F_{p^2}^{\times}$
induced from $\varrho^{12}\chi^{-6}$ is trivial on
$\Psi_{K/k}(Cl_k)$.
\end{enumerate}
\end{lem}

\begin{proof}

(1)
The character $\varrho^{12}\chi^{-6}$ is unramified outside $p$.
For $j=1,2$,
let $\mfP_j$ be the prime of $K$ above $\mfp_j$.
We have
$r(\mfP_j)^{12}(u)=\util^{-12}$
(see \cite[Proposition 4.8]{J})
and
$\chi_{\mfP_j}(u)=\util^{-2}$
for any $u\in\cO_{K,\mfP_j}^{\times}$, $j=1,2$.
Then $\varrho^{12}\chi^{-6}$ is unramified at $\mfP_1,\mfP_2$.

(2)
Assume $p\not\in\cP(\cA_{3,\cS})\cup\Set{l_{\mfq}|\mfq\in\cS}$.
Fix any $\mfq\in\cS$, and let $q=l_{\mfq}$.
Then $p\ne q$.
Let $\mfQ$ be the prime of $K$ above $\mfq$.
Then
$\N_{\mfQ}=\N_{\mfq}=q$
and
$\mfq\cO_K=\mfQ^2$.
We have
\begin{equation*}
\displaystyle\varrho(F_{\mfQ})^{24h}=\varrho^{12}(F_{\mfQ}^{2h})
=\varrho^{12}(\mfQ^{2h})=\varrho^{12}(\mfq^h\cO_K)
=\varrho^{12}(\beta\cO_K)
\end{equation*}
\begin{equation*}
=\varrho^{12}((1)_{\infty},(1)_p,(\beta)^{\infty,p})
=\varrho^{12}((\beta^{-1})_{\infty},(\beta^{-1})_p,(1)^{\infty,p})
=\varrho^{12}((\beta^{-1})_p,(1)^{p})
\end{equation*}
\begin{equation*}
=\prod_{j=1}^2 r(\mfP_j)^{12}(\beta^{-1})
=\prod_{j=1}^2(\beta\bmod{\mfp_j})^{12}
\equiv\Norm_{k/\Q}(\beta)^{12}\bmod{p}.
\end{equation*}
Here, 
$((1)_{\infty},(1)_p,(\beta)^{\infty,p})$, $((\beta^{-1})_{\infty},(\beta^{-1})_p,(1)^{\infty,p})$,
$((\beta^{-1})_p,(1)^p)$)
are the id\`{e}les of $K$
which are defined in the same way as in the proof of Proposition \ref{c=24}.
%
%
Since 
$\Norm_{k/\Q}(\beta)^{12}=q^{12h}$,
we have
$$\varrho(F_{\mfQ})^{24h}\equiv q^{12h}\bmod{p}.$$
On the other hand,
$$\varrho(F_{\mfQ})\equiv\pi\bmod{\mfp_0}$$ by (\ref{rhopip}).
Then
$q^{12h}\equiv\pi^{24h}\bmod{\mfp_0}$,
and so
$p$ divides 
$\Norm_{\Q(\pi)/\Q}(\pi^{24h}-q^{12h})
=-q^{12h}((\pi^{24h}+\pib^{24h})-2q^{12h})$.
Hence $p$ divides
$(\pi^{24h}+\pib^{24h})-2q^{12h}\in\cA_{3,\mfq}\subseteq\cA_{3,\cS}$.
Since $p\not\in\cP(\cA_{3,\cS})$, we have
$\pi^{24h}+\pib^{24h}=2q^{12h}$.
Since $\pi^{24h}\pib^{24h}=q^{24h}$, we have
$\pi^{24h}=\pib^{24h}=q^{12h}$.
Then
$\pi^{24}=q^{12}$
by the same argument as in the proof of Lemma \ref{rho12chi-6}(2).
Therefore
$$\varrho^{12}(\mfq\cO_K)=\varrho^{12}(\mfQ^2)=\varrho^{24}(F_{\mfQ})
\equiv\pi^{24}=q^{12}
\equiv\chi(F_{\mfQ})^{12}$$
$$=\chi^6(\mfQ^2)=\chi^6(\mfq\cO_K)\bmod{p}.$$
\end{proof}

\noindent
\textit{Proof of Proposition \ref{psplit}}.
\ \ 
Assume 
$\displaystyle p\not\in\Set{2}\cup\bigcap_{\cS\in\cT}
\Bigl(\cP(\cA_{3,\cS})\cup\Set{l_{\mfq}|\mfq\in\cS}\Bigr)$.
Then there is an element $\cS\in\cT$ such that
$p\not\in\cP(\cA_{3,\cS})\cup\Set{l_{\mfq}|\mfq\in\cS}$.
Fix such $\cS$.
By Lemma \ref{p=2or}, we have $p\equiv 1\bmod{4}$.
Let $q<\frac{p}{4}$ be any prime number that is not inert in $k$.
Let $\mfq$ be a prime of $k$ above $q$.
By replacing $K_0$ if necessary, we may assume that
any prime in $\cS\cup\Set{\mfp_1,\mfp_2,\mfq}$ is ramified in $K$.
Let $\mfQ$ be the prime of $K$ above $\mfq$.
Then $\N_{\mfQ}=\N_{\mfq}=q$
and
$\mfq\cO_K=\mfQ^2$.
By Lemma \ref{rho12chi-6'}(2), we have
$$\alpha(F_{\mfQ})^{12}=\alpha^6(\mfQ^2)=\alpha^6(\mfq\cO_K)
=\varrho^{12}(\mfq\cO_K)\chi^{-6}(\mfq\cO_K)=1.$$
Since $p\equiv 1\bmod{4}$,
we have $\alpha^{\frac{p+1}{2}}=\chi^{\frac{(1-p)(3+p)}{4}}=\one$.
Then
$\alpha(F_{\mfQ})^3=1$
and
$$\text{$\alpha(F_{\mfQ})+\alpha(F_{\mfQ})^{-1}=-1$ or $2$.}$$
By the same argument as in the proof of Proposition \ref{c=12(p+1)}, we have
$$\pi^2+\pib^2\equiv 
q\left(\frac{q}{p}\right)(\alpha(F_{\mfQ})+\alpha(F_{\mfQ})^{-1})\bmod{p}$$
and
$\left(\frac{q}{p}\right)=-1$.
If $q>2$, then $\left(\frac{p}{q}\right)=-1$,
and so $q$ does not split in $\Q(\sqrt{p})$.
Therefore $p\in\cN^{prime}$.
\qed

By Propositions \ref{cbaai}--\ref{c=12(p+1)} and \ref{psplit}, we obtain (\ref{P_A(k)}).

\section{The case where $B\otimes_{\Q}k\cong\M_2(k)$ and $p$ is inert in $k$}
\label{sec:ksplitsB}

In this section, we assume 
\begin{itemize}
\item
$B\otimes_{\Q}k\cong\M_2(k)$, and
\item
$p$ is inert in $k$.
\end{itemize}
This case was done by Jordan in \cite{J}.
We repeat the argument and slightly modify the exceptional set of prime numbers.
In this case, we have $K=k$, $\mfP=\mfp=p\cO_k$ and $\mfp\not\in\cS_0$.
As already noted, $p$ never splits in $k$ when $B\otimes_{\Q}k\cong\M_2(k)$.

%

First, we classify the character $r(\mfp)$.

\begin{prop}
\label{cbaai'}
If $p>7$,
then there is an integer
$b\in\Set{12,12p,4(p+2),4(2p+1),6(p+1)}$
such that
$$r(\mfp)^{12}(u)=\util^{-b}$$
for any
$u\in\cO_{k,\mfp}^{\times}$.
Furthermore, we can take
\begin{equation*}
\begin{cases}
b\in\Set{12,12p,6(p+1)} &\text{if $p\not\equiv 1\bmod{3}$, and}\\
b\in\Set{12,12p,4(p+2),4(2p+1)} &\text{if $p\not\equiv 1\bmod{4}$.}
\end{cases}
\end{equation*}

\end{prop}

\begin{proof}

It follows from \cite[Corollary 4.13 and Remarks 4.10, 4.14]{J}.
\end{proof}



Suppose that we are in the situation of Proposition \ref{cbaai'}.
We take $b$ to satisfy
\begin{equation*}
\begin{cases}
b\in\Set{12,12p,6(p+1)} &\text{if $p\not\equiv 1\bmod{3}$, and}\\
b\in\Set{12,12p,4(p+2),4(2p+1)} &\text{if $p\not\equiv 1\bmod{4}$.}
\end{cases}
\end{equation*}
%
%
We study the cases [$b=12$ or $12p$], [$b=4(p+2)$ or $4(2p+1)$]
and [$b=6(p+1)$] separately.

[Case $b=12$ or $12p$].

\begin{prop}
\label{b=12}
We have
$\displaystyle p\in\bigcap_{\mfq\in\cS_0}\cP(\cA'_{1,\mfq})$.

\end{prop}

\begin{proof}

Fix any $\mfq\in\cS_0$.
Then $\mfp\ne \mfq$.
By replacing $\varrho$ with $\varrho^p$ if necessary, we may assume $b=12$.
Then
$$r(\mfp)^{12}(u)=\util^{-12}$$
for any
$u\in\cO_{k,\mfp}^{\times}$.

Let $q=l_{\mfq}$.
Then $\N_{\mfq}=q$.
Note that $p\ne q$.
We claim
\begin{enumerate}[\upshape (i)]
\item
$a(F_{\mfq}^{12h_{\mfq}})\equiv\Tr_{k/\Q}(\bq^{12})\bmod{p}$,
and
\item
$a(F_{\mfq}^{12h_{\mfq}})\ne\Tr_{k/\Q}(\bq^{12})$,
\end{enumerate}
which are proved as follows.


(i)
We have
\begin{equation*}
\varrho(F_{\mfq}^{12h})=\varrho^{12}(\mfq^h)
=\varrho^{12}(\beta\cO_k)
\end{equation*}
\begin{equation}
\label{rhoF12h}
=\varrho^{12}((1)_{\infty},(1)_p,(\beta)^{\infty,p})
=\varrho^{12}((\beta^{-1})_{\infty},(\beta^{-1})_p,(1)^{\infty,p})
\end{equation}
\begin{equation*}
=\varrho^{12}((\beta^{-1})_p,(1)^p)
=r(\mfp)^{12}(\beta^{-1}).
\end{equation*}
%
Here, 
$((1)_{\infty},(1)_p,(\beta)^{\infty,p})$,
$((\beta^{-1})_{\infty},(\beta^{-1})_p,(1)^{\infty,p})$,
$((\beta^{-1})_p,(1)^p)$
are the id\`{e}les of $k$ which are defined in the same way as in the proof of Proposition
\ref{c=24}.
%
%
Then
$\varrho(F_{\mfq}^{12h})=(\betatil^{-1})^{-12}=\betatil^{12}$.
%
By \cite[Proposition 5.3]{J}, we obtain
$$a(F_{\mfq}^{12h})
\equiv\Tr_{\F_{p^2}/\F_p}(\varrho(F_{\mfq}^{12h}))
=\Tr_{\F_{p^2}/\F_p}(\betatil^{12})
\equiv\Tr_{k/\Q}(\beta^{12})\bmod{p}.$$

(ii)
Suppose $a(F_{\mfq}^{12h})=\Tr_{k/\Q}(\beta^{12})$.
Since $\beta\betab=\N_{\mfq}^h=q^h$, the roots of
\begin{equation}
\label{eq1'}
X^2-\Tr_{k/\Q}(\beta^{12})X+q^{12h}=0
\end{equation}
are $\beta^{12},\betab^{12}$.
On the other hand,
the roots of
$X^2-a(F_{\mfq})X+q=0$
are $\pi,\pib$.
We have $\pi\pib=q$ and
$a(F_{\mfq}^{12h})=\pi^{12h}+\pib^{12h}$.
Then $\pi^{12h},\pib^{12h}$ are the roots of (\ref{eq1'})
since $a(F_{\mfq}^{12h})=\Tr_{k/\Q}(\beta^{12})$.
Hence 
$\Set{\beta^{12},\betab^{12}}=\Set{\pi^{12h},\pib^{12h}}$,
and so
$$\Q(\pi)\supseteq\Q(\pi^{12h})=\Q(\beta^{12})\subseteq\Q(\beta)\subseteq k.$$
We prove $\beta^{12}\not\in\Q$.
Assume otherwise, i.e., $\beta^{12}\in\Q$.
Then
$q^{12h}
=\beta^{12}\betab^{12}=(\beta^{12})^2$, and so
$\beta^{12}=\pm q^{6h}$.
Hence
$\mfq^{12h}=\beta^{12}\cO_k=q^{6h}\cO_k$ and
$\mfq^2=q\cO_k$.
This implies that $\mfq$ is ramified in $k/\Q$,
which contradicts $\mfq\in\cS_0$.
Therefore $\beta^{12}\not\in\Q$.
Since $[\Q(\pi):\Q]=[k:\Q]=2$,
we have
$$\Q(\pi)=\Q(\pi^{12h})=\Q(\beta^{12})=\Q(\beta)=k.$$
Then $\pi\in k$.
Since $\pi\pib=q$, we have
$\mfq=\pi\cO_k$ or $\pib\cO_k$.
This implies that $\mfq$ is principal, which contradicts $\mfq\in\cS_0$.
Then (ii) has been proved.

By (i) and (ii), $p$ divides the non-zero integer
$a(F_{\mfq}^{12h})-\Tr_{k/\Q}(\beta^{12})$.
Since $a(F_{\mfq}^{12h})=\pi^{12h}+\pib^{12h}\in\cC'_{\mfq}$, we have
$a(F_{\mfq}^{12h})-\Tr_{k/\Q}(\beta^{12})\in\cA'_{1,\mfq}$.
Then $p\in\cP(\cA'_{1,\mfq})$.
Therefore $\displaystyle p\in\bigcap_{\mfq\in\cS_0}\cP(\cA'_{1,\mfq})$.
\end{proof}

[Case $b=4(p+2)$ or $4(2p+1)$].

\begin{prop}
\label{c=8(p+2)}
We have
$\displaystyle p\in\bigcap_{\mfq\in\cS_0}\cP(\cA'_{2,\mfq})$.

\end{prop}

\begin{proof}

Fix any $\mfq\in\cS_0$.
By replacing $\varrho$ with $\varrho^p$ if necessary, we may assume $b=4(p+2)$.
Then
$$r(\mfp)^{12}(u)=\util^{-4(p+2)}$$
for any
$u\in\cO_{k,\mfp}^{\times}$.

Let $q=l_{\mfq}$.
We claim
\begin{enumerate}[\upshape (i)]
\item
$a(F_{\mfq}^{12h_{\mfq}})\equiv q^{4h_{\mfq}}\Tr_{k/\Q}(\bq^4)\bmod{p}$,
and
\item
$a(F_{\mfq}^{12h_{\mfq}})\ne q^{4h_{\mfq}}\Tr_{k/\Q}(\bq^4)$.
\end{enumerate}


(i)
By (\ref{rhoF12h}), we have 
$$\varrho(F_{\mfq}^{12h})=r(\mfp)^{12}(\beta^{-1})
=(\betatil^{-1})^{-4(p+2)}
=\betatil^{4(p+2)}.$$
Since $\betab^p\equiv\beta\bmod{\mfp}$,
$\beta^p\equiv\betab\bmod{\mfp}$
and
$\beta\betab=q^h$,
we have
$$a(F_{\mfq}^{12h})
\equiv\Tr_{\F_{p^2}/\F_p}(\varrho(F_{\mfq}^{12h}))
=\Tr_{\F_{p^2}/\F_p}(\betatil^{4(p+2)})
\equiv\Tr_{k/\Q}(\beta^{4(p+2)})$$
$$=\beta^{4p+8}+\betab^{4p+8}
\equiv\betab^4\beta^{8}+\beta^4\betab^{8}
=\beta^4\betab^4(\beta^4+\betab^4)
=q^{4h}\Tr_{k/\Q}(\beta^4)\bmod{p}.$$

(ii)
Suppose $a(F_{\mfq}^{12h})=q^{4h}\Tr_{k/\Q}(\beta^4)$.
Since $\beta^4\betab^4=q^{4h}$, the roots of
\begin{equation}
\label{eq2'}
X^2-\Tr_{k/\Q}(\beta^4)X+q^{4h}=0
\end{equation}
are $\beta^4,\betab^4$.
On the other hand,
the roots of
$X^2-a(F_{\mfq})X+q=0$
are $\pi,\pib$.
Then
$\pi^{12h}+\pib^{12h}=a(F_{\mfq}^{12h})=q^{4h}\Tr_{k/\Q}(\beta^4)$
and
$\pi\pib=q$.
Hence
$\frac{\pi^{12h}}{q^{4h}}+\frac{\pib^{12h}}{q^{4h}}=\Tr_{k/\Q}(\beta^4)$
and
$\frac{\pi^{12h}}{q^{4h}}\cdot\frac{\pib^{12h}}{q^{4h}}=q^{4h}$.
Then
$\frac{\pi^{12h}}{q^{4h}},\frac{\pib^{12h}}{q^{4h}}$
are the roots of (\ref{eq2'}).
Hence
$\Set{\beta^4,\betab^4}=\Set{\frac{\pi^{12h}}{q^{4h}},\frac{\pib^{12h}}{q^{4h}}}$,
and so
$$\Q(\pi)\supseteq\Q(\pi^{12h})=\Q(\beta^4)\subseteq\Q(\beta)\subseteq k.$$
%
Assume $\beta^4\in\Q$.
Then $q^{4h}=\beta^4\betab^4=(\beta^4)^2$, and so $\beta^4=\pm q^{2h}$.
Hence $\mfq^{4h}=\beta^4\cO_k=q^{2h}\cO_k$
and
$\mfq^2=q\cO_k$.
This contradicts $\mfq\in\cS_0$.
Therefore $\beta^4\not\in\Q$.
Then
$$\Q(\pi)=\Q(\pi^{12h})=\Q(\beta^4)=\Q(\beta)=k,$$
and so $\pi\in k$.
Since $\pi\pib=q$, we have
$\mfq=\pi\cO_k$ or $\pib\cO_k$.
This contradicts $\mfq\in\cS_0$.
Then (ii) has been proved.

Since $a(F_{\mfQ}^{12h})-q^{4h}\Tr_{k/\Q}(\beta^4)\in\cA'_{2,\mfq}$,
we conclude $p\in\cP(\cA'_{2,\mfq})$.
\end{proof}

[Case $b=6(p+1)$].
In this case, we have $p\equiv 1\bmod{4}$.

\begin{prop}
\label{b=6(p+1)}
We have
$\displaystyle p\in\left(\bigcap_{\cS\in\cT}
\cP(\cA'_{3,\cS})\right)
\cup\cN^{prime}$.
\end{prop}

Fix any $\cS\in\cT$.
First, we prove:

\begin{lem}
\label{rho24chi-12}
\begin{enumerate}[\upshape (1)]
\item
The character
$\varrho^{12}\chi^{-6}:\G_k\longrightarrow\F_{p^2}^{\times}$
is unramified everywhere.
\item
If $p\not\in\cP(\cA'_{3,\cS})$, then the character
$\varrho^{24}\chi^{-12}:\G_k\longrightarrow\F_{p^2}^{\times}$ is trivial.
\end{enumerate}
\end{lem}

\begin{proof}

(1)
It suffices to prove that $\varrho^{12}\chi^{-6}$
is unramified at $\mfp$ ($=p\cO_k$).
We have
$$r(\mfp)^{12}(u)=\util^{-6(p+1)}
=(\util^{p+1})^{-6}
=\Norm_{\F_{p^2}/\F_p}(\util)^{-6}$$
and
$\chi_{\mfp}(u)=\Norm_{\F_{p^2}/\F_p}(\util)^{-1}$
for any
$u\in\cO_{k,\mfp}^{\times}$
(\cf \cite[Proof of Proposition 4.8]{J}).
Hence
$r(\mfp)^{12}(u)\chi_{\mfp}^{-6}(u)=1$
for any $u\in\cO_{k,\mfp}^{\times}$.
Therefore $\rho^{12}\chi^{-6}$ is unramified at $\mfp$.

(2)
Assume $p\not\in\cP(\cA'_{3,\cS})$.
Fix any $\mfq\in\cS$, and let $q=l_{\mfq}$.
Then $\N_{\mfq}=q$ and $p\ne q$.
By (\ref{rhoF12h}), we have
$$\varrho(F_{\mfq}^{12h})=r(\mfp)^{12}(\beta^{-1})
=(\betatil^{-1})^{-6(p+1)}
=\betatil^{6(p+1)},$$
and so
$$a(F_{\mfq}^{12h})
\equiv\Tr_{\F_{p^2}/\F_p}(\varrho(F_{\mfq}^{12h}))
=\Tr_{\F_{p^2}/\F_p}(\betatil^{6(p+1)})
=\betatil^{6(p+1)}+\betatil^{6(p^2+p)}$$
$$=2\betatil^{6(p+1)}
\equiv 2\beta^{6}\betab^{6}
=2q^{6h}\bmod{p}.$$
%
Suppose $a(F_{\mfq}^{12h})\ne 2q^{6h}$.
Then $p$ divides the non-zero integer
$a(F_{\mfq}^{12h})-2q^{6h}\in\cA'_{3,\mfq}\subseteq\cA'_{3,\cS}$,
and so
$p\in\cP(\cA'_{3,\cS})$.
This contradicts the assumption $p\not\in\cP(\cA'_{3,\cS})$.
Hence $a(F_{\mfq}^{12h})=2q^{6h}$.
Then
$\pi^{12h}+\pib^{12h}=a(F_{\mfq}^{12h})=2q^{6h}$
and
$\pi^{12h}\pib^{12h}=q^{12h}$.
Therefore
$\pi^{12h}=\pib^{12h}=q^{6h}$.
Then 
$\pi^{24h}=q^{12h}$, 
and so
$\pi^{24}=q^{12}$
as seen in the proof of Lemma \ref{rho12chi-6}(2).
Note that the case
$\pi^{12}=-q^{6}$ ($\ne q^6$)
can occur (e.g., $q=2$ and $\pi=1+\sqrt{-1}$).

%

By (\ref{rhopip}), we have
$\varrho(F_{\mfq})\equiv\pi\bmod{\mfp_0}$.
Then
$$\varrho^{24}(\mfq)=\varrho^{24}(F_{\mfq})
\equiv\pi^{24}=q^{12}\equiv\chi(F_{\mfq})^{12}
=\chi(\mfq)^{12}\bmod{p}.$$
By (1), $\varrho^{24}\chi^{-12}$ induces a character
$Cl_k\longrightarrow\F_{p^2}^{\times}$.
This character is trivial since $\cS$ generates $Cl_k$.
Then the assertion follows.
\end{proof}

\noindent
\textit{Proof of Proposition \ref{b=6(p+1)}}.
\ \ 
Assume
$p\not\in\displaystyle \bigcap_{\cS\in\cT}\cP(\cA'_{3,\cS})$.
For any prime number $q<\frac{p}{4}$ that is not inert in $k$,
we prove $\left(\frac{q}{p}\right)=-1$ as follows.

There is an element $\cS\in\cT$ such that $p\not\in\cP(\cA'_{3,\cS})$.
Fix such $\cS$.
Let $\mfq$ be a prime of $k$ above $q$.
Note that $p>4q>q$.
Since $q$ is not inert in $k$, we have
$\N_{\mfq}=q$.
%
By Lemma \ref{rho24chi-12}(2), we have
$$\alpha(F_{\mfq})^{12}
=\alpha^{12}(\mfq)
=\varrho^{24}(\mfq)\chi^{-12}(\mfq)
=1.$$
%
Since $p\equiv 1\bmod{4}$, we have
$\alpha^{\frac{p+1}{2}}=\chi^{\frac{(1-p)(3+p)}{4}}=\one$.
Then
$\alpha(F_{\mfq})^3=1$,
and so
$$\text{$\alpha(F_{\mfq})+\alpha(F_{\mfq})^{-1}=-1$ or $2$.}$$
%
%
By the same argument as in the proof of Proposition \ref{c=12(p+1)},
we have 
$$\pi^2+\pib^2\equiv
q\left(\frac{q}{p}\right)(\alpha(F_{\mfq})+\alpha(F_{\mfq})^{-1})\bmod{p}$$
and 
$\left(\frac{q}{p}\right)=-1$.
%

Then we conclude $p\in\cN^{prime}$.
%
\qed

By Propositions \ref{cbaai'}--\ref{b=6(p+1)}, we obtain (\ref{P_A'(k)}).
Then Theorem \ref{mainthm'} has been proved.

\subsection*{Acknowledgements}

This work was supported by JSPS KAKENHI Grant Numbers 
JP25800025, JP16K17578, JP21K03187
and
Research Institute for Science and Technology of Tokyo Denki 
University Grant Numbers Q16K-06, Q20K-01 / Japan.

\def\bibname{References}

\end{document}